\documentclass[11pt]{amsart}
\usepackage{array,amsmath, enumerate, color, url}
\usepackage{amssymb, amsaddr}
\usepackage{graphicx,subfigure}

\makeatletter
 \usepackage{fullpage} 
 \usepackage{amsthm}

\usepackage{fullpage}

\makeatother

\begin{document}
\newcommand{\n}{\noindent}
\newtheorem{thm}{Theorem}[section]
\newtheorem{lem}[thm]{Lemma}
\newtheorem{prop}[thm]{Proposition}
\newtheorem{cor}[thm]{Corollary}
\newtheorem{con}[thm]{Conjecture}
\newtheorem{claim}[thm]{Claim}
\newtheorem{obs}[thm]{Observation}
\newtheorem{definition}[thm]{Definition}
\newtheorem{example}[thm]{Example}
\newtheorem{rmk}[thm]{Remark}
\newcommand{\di}{\displaystyle}
\def\dfc{\mathrm{def}}
\def\cF{{\cal F}}
\def\cH{{\cal H}}
\def\cT{{\cal T}}
\def\C{{\mathcal C}}
\def\cA{{\cal A}}
\def\cB{{\mathcal B}}
\def\P{{\mathcal P}}
\def\Q{{\mathcal Q}}
\def\cP{{\mathcal P}}
\def\cp{\alpha'}
\def\Frk{F_k^{2r+1}}

\title{A relaxation of the Bordeaux Conjecture}
\author{Runrun Liu$^{\dagger}$ and Xiangwen Li$^{\dagger}$ and
Gexin Yu$^{\dagger}\ ^{\ddagger}$}\thanks{The second author was supported by the
Natural Science Foundation of China (11171129)  and by Doctoral Fund of Ministry of Education of China (20130144110001);
 The third author's research was supported in part by NSA grant H98230-12-1-0226.}

\address{$^{\dagger}$ Department of Mathematics, Huazhong Normal University, Wuhan, 430079, China}

\address{$^{\ddagger}$ Department of Mathematics, The College of William and Mary,
Williamsburg, VA, 23185, USA.  }

\email{xwli68@mail.ccnu.edu.cn, gyu@wm.edu}

\date{\today}
\maketitle

\begin{abstract}
A $(c_1,c_2,...,c_k)$-coloring of a graph $G$ is a mapping $\varphi:V(G)\mapsto\{1,2,...,k\}$ such that for every $i,1 \leq i \leq k$, $G[V_i]$ has maximum degree at most $c_i$, where $G[V_i]$ denotes the subgraph induced by the vertices colored $i$. Borodin and Raspaud conjecture that every planar graph with neither $5$-cycles nor  intersecting triangles is $3$-colorable. We prove in this paper that every planar graph with neither $5$-cycles nor intersecting triangles is (2,0,0)-colorable.
\end{abstract}

\section{Introduction}

It is well-known that the problem of deciding whether a planar graph is properly $3$-colorable is NP-complete.  Gr\"{o}tzsch \cite{G59} proved the famous theorem that every triangle-free planar graph is $3$-colorable.   A lot of research has been devoted to find sufficient conditions for a planar graph to be $3$-colorable, by allowing a triangle together with some other conditions.  One of such efforts is the following famous conjecture made by Steinberg  \cite{S76}.

\begin{con}[Steinberg, \cite{S76}]
All planar graphs with neither $4$-cycles nor $5$-cycles are $3$-colorable.
\end{con}

Some progresses have been made towards this conjecture, along two directions.  One direction was suggested by Erd\H{o}s to find a constant $c$ such that a planar graph without cycles of length from $4$ to $c$ is $3$-colorable.  Borodin, Glebov, Raspaud, and Salavatipour~\cite{BGRS05} showed that $c\le 7$.  For more results, see the recent nice survey by Borodin~\cite{B12}.

Another direction of relaxation of the conjecture is to allow some defects in the color classes.  A graph is {\em $(c_1, c_2, \cdots, c_k)$-colorable} if the vertex set can be partitioned into $k$ sets $V_1,V_2, \ldots, V_k$, such that for every $i: 1\leq i\leq k$ the subgraph $G[V_i]$ has maximum degree at most $c_i$.  Here, $c_i$ is the {\em deficiency} of color $i$.   Thus a $(0,0,0)$-colorable graph is properly $3$-colorable. 
Chang, Havet, Montassier, and Raspaud~\cite{CHMR11} proved that all planar graphs with neither $4$-cycles nor $5$-cycles are $(2,1,0)$-colorable and $(4,0,0)$-colorable.  In~\cite{HSWXY13, HY13, XMW12}, it is shown that planar graphs with neither $4$-cycles nor $5$-cycles are $(3,0,0)$- and $(1,1,0)$-colorable.

Havel \cite{H69} asked if each planar graph with a large enough minimum distance $d^{\bigtriangledown}$ between triangles is $3$-colorable.  This was resolved in a recent preprint of Dvo\v{r}\'ak, Kr\'al and Thomas~\cite{DKT09}.  Borodin and Raspaud in 2003 made the following Bordeaux Conjecture
, which has common features with Havel's (1969) and Steinberg's (1976) 3-color problems.

\begin{con}[Borodin and Raspaud, \cite{BR03}] \label{con1}
Every planar graph with $d^{\bigtriangledown}\ge 1$ and without $5$-cycles is $3$-colorable.
\end{con}

A relaxation of the Bordeaux Conjecture with $d^{\bigtriangledown}\ge 4$ was confirmed  by Borodin and Raspaud~\cite{BR03}, and the result was improved to $d^{\bigtriangledown}\ge 3$ by Borodin and Glebov  \cite{BG04} and, independently, by  Xu \cite{X07}.     Borodin and Glebov \cite{BG11} further improved the result to $d^\bigtriangledown\ge2$.

In terms of relaxed coloring,  Xu~\cite{X08} proved that all planar graphs with neither adjacent triangles nor  $5$-cycles are $(1,1,1)$-colorable, where two triangles are adjacent if they share an edge.

In this paper, we consider another relaxation of the Bordeaux Conjecture. Let $\mathcal{G}$ be the family of plane graphs with $d^{\bigtriangledown}\ge 1$ and without $5$-cycles.   Yang and Yerger~\cite{YY14} showed that
planar graphs in $\mathcal{G}$ are $(4,0,0)$- and $(2,1,0)$-colorable, but there is a flaw in one of their key lemmas (Lemma~2.4)\footnote{According to private communication with Yerger, they may have a way to fix the gap in their proofs.}.   In \cite{LLY05}, we showed that graphs in $\mathcal{G}$ are $(1,1,0)$-colorable. We prove the following result.

\begin{thm}\label{main-1}
Every planar graph in $\mathcal{G}$ is $(2,0,0)$-colorable.
\end{thm}

In fact, we will prove a stronger result. 
Let $G$ be a graph and $H$ be a subgraph of $G$.  We call $(G,H)$ {\em superextendable} if any $(2,0,0)$-coloring of $H$ can be extended to $G$ so that the vertices in $G-H$ have different colors from their neighbors in $H$;  in this case, we call $H$ a superextendable subgraph.

\begin{thm}\label{main}
Every triangle or $7$-cycle of a planar graph in $\mathcal{G}$ is superextendable.
\end{thm}

To see the truth of Theorem~\ref{main-1} by way of Theorem~\ref{main},  we may assume that the planar graph contains a triangle $C$ since $G$ is 3-colorable if $G$ has no triangle. Then color the triangle, and by Theorem~\ref{main}, the coloring of $C$ can be superextended to $G$. Thus, we get a coloring of $G$.

We will use a discharging argument to prove Theorem~\ref{main}, that is, we consider a minimal counterexample and assign an initial charge to each vertex and face so that the sum is $0$. We shall design some rules to redistribute the charges among vertices and faces so that  some local sparse structures appear, or otherwise all vertices and faces would have non-negative and at least one has a positive final charges. We will then show that the coloring outside the sparse structures can be extended to include all vertices in the graph (that is, the local structure is reducible), to reach a contradiction.

As pointed out in \cite{YY14}, as we may have $4$-cycles in the considered graphs, the proof is quite different from the previous known relaxations of the Steinberg's Conjecture in terms of relaxed coloring.

The paper is organized as follows.  In Section~\ref{prelim}, we introduce some notations used in the paper.  In Section~\ref{reduce}, we show the reducible structures useful in our proof.  In Section~\ref{discharging}, we show the discharging process to finish the proof.

\section{Preliminaries}\label{prelim}
In this section, we introduce some notations used in the paper.

Graphs mentioned in this paper are all simple.  A $k$-vertex ($k^+$-vertex, $k^-$-vertex) is a vertex of degree $k$ (at least $k$, at most $k$). The same notation will be applied to faces and cycles. We use $b(f)$ to denote the vertex set of a face $f$. We use $F(G)$ to denote the set of faces in $G$. An $(l_1, l_2, \ldots, l_k)$-face is a $k$-face $v_1v_2\ldots v_k$ with $d(v_i)=l_i$, respectively. A face $f$ is a pendant $3$-face of vertex $v$ if $v$ is not on $f$ but is adjacent to some $3$-vertex on $f$. The pendant neighbor of a $3$-vertex $v$ on a $3$-face is the neighbor of $v$ not on the $3$-face.

Let $C$ be a cycle of a plane graph $G$. We use $int(C)$ and $ext(C)$ to denote the sets of vertices located inside and outside $C$, respectively. The cycle $C$ is called a {\em separating cycle} if $int(C)\ne\emptyset\ne ext(C)$, and is called a {\em nonseparating cycle} otherwise. We still use $C$ to denote the set of vertices of $C$.

Let $S_1, S_2, \ldots, S_l$ be pairwise disjoint subsets of $V(G)$. We use $G[S_1, S_2, \ldots, ,S_l]$ to denote the graph obtained from $G$ by identifying all the vertices in $S_i$ to a single vertex for each $i\in\{1, 2, \ldots, l\}$.  Let $x(y)$ be the vertex obtained by identifying $x$ and $y$ in $G$.

A vertex $v$ is {\em properly colored} if all neighbors of $v$ have different colors from that of $v$. A vertex $v$ is {\em nicely colored} if it shares a color (say $i$) with at most max$\{s_i-1,0\}$ neighbors, where $s_i$ is the deficiency allowed for color $i$. Thus if a vertex $v$ is nicely colored by a color $i$ which allows deficiency $s_i>0$, then an uncolored neighbor of $v$ can be colored by $i$.

\section{Reducible configurations}\label{reduce}
Let $(G, C_0)$ be a minimum counterexample to Theorem~\ref{main} with minimum $\sigma(G)=|V(G)|+|E(G)|$, where $C_0$ is a triangle or a $7$-cycle in $G$ that is precolored.

The following are some simple observations about $(G,C_0)$.

\begin{prop}\label{prop3.1}
(a) Every vertex not on $C_0$ has degree at least $3$.\\
(b) Every vertex in $G$ can have at most one incident $3$-face.\\
(c) No $3$-face and $4$-face in $G$ can have a common edge.
\end{prop}

Similar to the lemmas in \cite{X08}, we show Lemmas~\ref{lem3.2} to ~\ref{lem3.6}, which hold for all superextenable $(c_1,c_2,c_3)$-coloring of $G\in\mathcal{G}$.  The proofs are similar to those of \cite{X08}, and for completeness, we include the proofs here. If $C_0$ is a separating cycle, then $C_0$ is superextendable in both $G-ext(C_0)$ and $G-int(C_0)$. Thus, $C_0$ is superextendable in $G$, contrary to the choice of $C_0$. Thus, we may assume that $C_0$ is the boundary of the outer face of $G$ in the rest of this paper.

\begin{lem}\label{lem3.2}
The graph $G$ contains neither separating triangles nor separating $7$-cycles.
\end{lem}

\begin{proof}
Let $C$ be a separating triangle or $7$-cycle in $G$. Then $C$ is inside of $C_0$. By the minimality of $G$, $(G- int(C), C_0)$ is superextendable, and after that, $C$ is colored. By the minimality of $G$ again,  $(C\cup int(C), C)$ is superextendable. Thus,  $(G,C_0)$ is superextendable,  a contradiction.
\end{proof}

\begin{lem}\label{lem3.3}
If $G$ has a separating $4$-cycle $C_1=v_1v_2v_3v_4v_1$, then $ext(C_1)=\{b,c\}$ such that $v_1bcv_1$ is a $3$-cycle. Furthermore, the $4$-cycle is a unique separating $4$-cycle.
\end{lem}

\begin{proof}
Suppose that the lemma is not true.  Let $G_1=G- int(C_1)$ and $G_2$ be the graph obtained from $G- ext(C_1)$ by substituting $v_1w_1w_2w_3v_2$ for $v_1v_2$.  Let $C_2=v_1w_1w_2w_3v_2v_3v_4v_1$.

Since $\sigma(G_1)<\sigma(G)$, $(G_1, C_0)$ is superextendable by the minimality of $G$. This means that $C_1$ can be colored and hence $C_2$ can be colored.  If $(G_2, C_2)$ is superextendable, then $(G,C_0)$ is superextendable, a contradiction.   Since $G\in\mathcal{G}$, no edge of $C_1$ is in any triangles. Therefore, $G_2\in\mathcal{G}$. We now show that $(G_2, C_2)$ is superextendable. For this goal, we need only to check that $\sigma(G_2)<\sigma(G)$. Note that $\sigma(G_2)=\sigma(G- ext(C_1))+6$.

If $|C_0|=7$, then $\sigma(C_0)-\sigma(C_0\cap C_1)\ge7$ as $C_1\ne C_0$, and thus $\sigma(G_2)=\sigma(G- ext(C_1))+6\le[\sigma(G)-(\sigma(C_0)-\sigma(C_0\cap C_1))]+6<\sigma(G)$. Thus, we may assume that $|C_0|=3$.

If $C_1\cap C_0=\emptyset$, then $G- int(C_1)-(E(C_0)\cup E(C_1))$ contains at least one edge as $G$ is connected, thus $\sigma(G_2)=\sigma(G- ext(C_1))+6\le(\sigma(G)-\sigma(C_0)-1)+6<\sigma(G)$. So we may further assume that $C_0\cap C_1\ne\emptyset$.

Since $G\in\mathcal{G}$, $|C_0\cap C_1|=1$. If $|ext(C_1)|\ge 3$, then $\sigma(G_2)=\sigma(G- ext(C_1))+6\le[\sigma(G)-((\sigma(C_0)-1)+2)]+6<\sigma(G)$. Therefore, $|ext(C_1)|=2$ and we obtain the desired structure in the lemma, a contradiction.

If $G$ contains another separating $4$-cycle, say $C'$, then $C'$ is a subgraph of $G- ext(C_1)$, but then $ext(C')$ contains more than two vertices, a contradiction.  So, $C_1$ is the unique separating $4$-cycle.
\end{proof}

\begin{lem}\label{lem3.4}
 If $x,y\in C_0$ with $xy\not\in E(C_0)$, then $xy\not\in E(G)$ and $N(x)\cap N(y)\subseteq C_0$.
\end{lem}

\begin{proof}
We may assume that $|C_0|=7$ as it is trivially true for $|C_0|=3$. Let $x, y$ be two vertices on cycle $C_0$ such that $xy\not\in E(C_0)$. Let $P$ be the shorter path on $C_0$ joining $x$ and $y$. Then $k=|E(P)|\in \{2,3\}$.

First we assume that $xy\in E(G)$. Since $G$ contains no $5$-cycle, $k=2$. Assume that $P=xvy$. Then, $xvy$ is a 3-face,  for otherwise it is a separating $3$-cycle, contradicting Lemma~\ref{lem3.2}, and $xy$ is not on any $4$-cycle. Let $H$ be the graph obtained from $G-\{v\}$ by inserting a vertex $v'$ into $xy$, where the broken edges and vertex $v$ are not in $H$. Then, $H\in\mathcal{G}$, $\sigma(H)=\sigma(G)-1$, and hence in $H$, $(C_0-\{v\})\cup xv'y$ is superextendable. But this means $(G, C_0)$ is superextendable, a contradiction. Therefore, $xy\not\in E(G)$.

Next, we assume that $u\in (N(x)\cap N(y))- C_0$. Again, since $G$ has no $5$-cycle, $k=2$ and  let $P=xvy$. Since $G\in\mathcal{G}$, $N(u)\cap C_0=\{x,y\}$. By Lemmas~\ref{lem3.2} and ~\ref{lem3.3}, both $xvyux$ and $C_0-v+u$ are facial cycles. Thus, $d(u)=2$, a contradiction to Proposition~\ref{prop3.1}(a).
\end{proof}

For convenience, let $F_k=\{f: \text{ $f$ is a $k$-face and } b(f)\cap C_0=\emptyset\}$,  $F_k'=\{f:  \text{ $f$ is a $k$-face and } |b(f)\cap C_0|=1\}$, and $F_k''=\{f:  \text{ $f$ is a $k$-face and }  |b(f)\cap C_0|=2\}$.

\begin{lem}\label{lem3.5}
Suppose that $f=v_1v_2v_3v_4$ is a $4$-face and $v_1\in C_0$. Then, $v_3\not\in C_0$. Moreover, $|N(v_3)\cap C_0|=1$ if $f\in F_4''$, and $|N(v_3)\cap C_0|=0$ if $f\in F_4'$.
\end{lem}
\begin{proof}
 Suppose on the contrary that  $v_3\in C_0$. By Lemma~\ref{lem3.4}, $v_2$ and $v_4$ are both in $C_0$. This implies that $|C_0|=7$ and $C_0$ has a chord, contrary to Lemma~\ref{lem3.4}.

Suppose first that $f\in F_4''$ and $b(f)\cap C_0=\{v_1, v_2\}$. By Lemma~\ref{lem3.4}, $v_3, v_4\not\in C_0$.  If $x\in (N(v_3)\cap C_0-v_2)$, then by Lemma~\ref{lem3.4}, $v_2x\in E(C_0)$, but then $v_1v_2xv_3v_4v_1$ is a $5$-cycle, a contradiction.  Therefore, $N(v_3)\cap C_0=\{v_2\}$.

Next,  suppose  otherwise that $f\in F_4'$, and $v_3$ has a neighbor, say $x$, in $C_0$. As $G$ has no $5$-cycle,  $|C_0|=7$. Let $C_0=v_1u_1u_2...u_6v_1$. We may assume that $x\in\{u_4,u_5,u_6\}$ by symmetry. If $x=u_4$, then $v_1v_4v_3u_4u_3u_2u_1v_1$ is a separating $7$-cycle. If $x=u_5$, then $v_1v_2v_3u_5u_6v_1$ is a $5$-cycle, a contradiction. If $x=u_6$, then $v_1v_2v_3u_6v_1$ is a separating $4$-cycle such that $\{u_1, u_2, u_3, u_4, u_5\}$ is outside of this separating 4-cycle, contrary to Lemma~\ref{lem3.3}.  Therefore, $|N(v_3)\cap C_0|=0$.
\end{proof}

\begin{lem}\label{lem3.6}
Let $u, w$ be non-consecutive vertices on a $4$-face. If at most one of $u$ and $w$ is incident to a triangle, then $G[\{u,w\}]\in \mathcal{G}$.
\end{lem}

\begin{proof}
Suppose that $f=uvwx$. By Lemma~\ref{lem3.5}, we may assume that $w,x\not\in C_0$.

Since $G\in\mathcal{G}$, $G$ has no $3$-path joining $u$ and $w$, thus no new triangle can be obtained from the identification of $u$ and $w$. Since at most one vertex in $\{u,w\}$ is incident to a triangle, the identification of $u$ and $w$ produces no intersecting triangles. If $G[\{u,w\}]$ has a $5$-cycle, then $G$ has a $5$-path $P'$ joining $u$ and $w$. If one of $v$ and $x$ is in $P'$, then $b(f)\cup P'$ has a $5$-cycle, a contradiction. So, $v,x\not\in V(P')$, and hence either $P'\cup uvw$ or $P'\cup uxw$ is a separating $7$-cycle; both contradict Lemma \ref{lem3.2}. Therefore, $G[\{u,w\}]\in \mathcal{G}$.
\end{proof}

For convenience, let $f=v_1v_2...v_k$ have corresponding degrees $(d_1,d_2,...d_k)$.

\begin{lem}\label{lem3.7}
Let $f=uvwx$ be a $4$-face in $F_4\cup F_4'$. Then (1) if $b(f)\cap C_0=\{u\}$,  then each of $u$ and $w$ is incident to  a triangle. (2)  if $f\in F_4$ is a $(4^-,3^+,4^-,3^+)$-face, then each of $v$ and $x$ is incident to a triangle.  In particular, there is no $(4^-,3,4^-,3^+)$-face in $F_4$.
\end{lem}

\begin{proof}
(1) Suppose on the contrary that at most one of $u$ and $w$ is incident to a triangle. By Lemma~\ref{lem3.6}, $G[\{u,w\}]\in \mathcal{G}$.
By Lemma~\ref{lem3.5}, $|N(w)\cap C_0|=0$. Since $\sigma(G[\{u,w\}])=\sigma(G)-3$, $(G[\{u,w\}], C_0)$ is superextendable such that the color of $u(w)$ is different from $v$ and $x$.
But then $(G, C_0)$ is superextendable, by coloring $u$ and $w$ with the color of $u(w)$ and preserving the colors of the other vertices, a contradiction.

(2) Suppose to the contrary that at most one of $v$ and $x$ is incident to a triangle. By Lemma~\ref{lem3.6}, $G[\{v,x\}]\in\mathcal{G}$. Then $(G[\{v,x\}], C_0)$ is superextendable. Color $v, x$ with the color of $v(x)$ and preserve the colors of the other vertices. We obtain a coloring of $(G,C_0)$, unless $v(x)$ is colored with $1$ and $u$ (or $w$) is colored with $1$ as well and one of the other two neighbors of $u$ (or $w$) is colored with $1$. Note that $v$ and $x$ have no common neighbors other  than $u$ and $w$ by Lemma~\ref{lem3.3}.  In this case, we recolor $u$ (or $w$)  properly and get a coloring of $(G, C_0)$.
\end{proof}

\begin{lem}\label{lem3.8}
Every $3$-vertex in $int(C_0)$ has either a neighbor on $C_0$ or a $5^+$-neighbor.
\end{lem}

\begin{proof}
Let $v\in int(C_0)$ be a $3$-vertex with no neighbor on $C_0$. If all neighbors of $v$ have degree at most $4$,  then $(G- v, C_0)$ is superextendable by the minimality of $G$. We may assume that all neighbors of $v$ are colored differently and $u$ be the neighbor of $v$ that is colored with $1$. Then either two neighbors of $u$ are colored with $1$, or $u$ is nicely colored. In the former case, we recolor $u$ with the color not in its neighbors and color $v$ with $1$, and in the latter case, we color $v$ with $1$, a contradiction.
\end{proof}

We could say more on the degrees of the neighbors of a $3$-vertex on a triangle. For a 3-vertex $u$, let $u'$ be the pendant neighbor of $u$ on a $3$-face $f=uvw$.

\begin{lem}\label{lem3.9}
Let $f=uvw$ be a $(3,3,5^-)$-face in $G$ with $d(u)=d(v)=3$. If $(b(f)\cup \{u'\})\cap C_0=\emptyset$, then $d(u')\ge 5$.
\end{lem}
\begin{proof}
The result is true for all $(3,3,4^-)$-faces in $F_3$  by Lemma~\ref{lem3.8}. So we may assume $d(w)=5$ and $d(u')\le 4$.  By the minimality of $G$, $(G- \{u,v\}, C_0)$ is superextendable. Properly color $v$, and $u$ cannot be properly colored only if $u', v, w$ are colored differently.
Note that $d(u')\le 4$. If $u'$ is colored with $1$, then either $u'$ is nicely colored  or two neighbors of $u'$ are colored
with $1$. In the former case,  we color $u$ with $1$; in the latter case,  we color $u'$ with the color not in
its neighbors and color $u$ with $1$.  If $v$ is colored with $1$, then we color $u$ with $1$ as well.
So we may assume that $w$ is colored with $1$, then either $w$ is nicely colored or two of the three other neighbors of $w$ other than $u,v$ are colored with $1$. In the former case,  we color $u$ with $1$; in the latter case, we recolor $w$ properly and color $u$ and $v$ with $1$.
\end{proof}

We define some special faces from $F_3$. First of all, $(3,4,4)$-faces and $(3,3,5^-)$-faces in $F_3$ are {\em special}.
Then we use a recursive method to define special $(3,5,5)$-faces. The initial special $(3,5,5)$-faces are those $(3,5,5)$-faces whose two $5$-vertices have six pendant $(3,3,5^-)$-faces or $(3,4,4)$-faces altogether;  then a $(3,5,5)$-face is {\em special} if the two $5$-vertices have six pendant $(3,3,5^-)$-faces, or $(3,4,4)$-faces, or initial or subsequent special $(3,5,5)$-faces altogether. Clearly, special $(3, 5, 5)$-faces are well-defined.
We call a $3$-face {\em special} if it is a $(3,3,5^-)$-face, or a $(3,4,4)$-face, or a special $(3,5,5)$-face.

The following is a technical lemma which we will use many times in the proofs of later lemmas.

\begin{lem}\label{lem3.10}
Let $f=uvw$ be a special $3$-face with $d(u)=3$ and $u'\not\in C_0$. Then a desired coloring of $(G-\{u, u'\}, C_0)$ can be extended to the desired coloring of $G-u'$ such that $u$ is colored with $1$.
\end{lem}

\begin{proof}
Let $f$ be a $(3,3,5^-)$-face. Note that $G-\{u, u'\}$ is $(2, 0, 0)$-colorable. If $w$ is not colored with $1$, we can color $u$ with $1$. Thus, we may assume that $w$ is colored with $1$. If $v$ is colored with $1$, then $w$ has at most one neighbor (other than $u$ and $v$) which was colored with $1$.  In this case,  we recolor $v$ properly and then color $u$ with $1$. Thus, assume further that $v$ is not colored with $1$. The vertex $u$ cannot be colored with $1$ if and only if $w$ has two neighbors (other than $u$ and $v$) which are colored with $1$. In this case, since $d(w)\leq 5$, $w$ can be nicely colored with 2 or 3 and then we recolor $v$ properly and color $u$ with $1$.

Let $f$ be a $(3,4,4)$-face. If $u$ cannot be colored with $1$,  then $w$ or $v$ is colored with $1$. If  both $w$ and $v$ are colored with $1$, then either $v$ or $w$, say $v$,  has a neighbor (other than $u$ and $w$) colored with $1$. In this case, we recolor $v$ properly and color $u$ with color $1$. If $v$ is colored with color $1$ and $w$ is not colored with $1$, then $v$ has two neighbors (other than $u$ and $w$) colored with $1$. Then we recolor $v$ properly and color $u$ with $1$.

Let $f=uvw$ be a special $(3,5,5)$-face.  Assume first that $f$ is an initial special $(3,5,5)$-face, that is, its two $5$-vertices have six pendant $(3,3,5^-)$ or $(3,4,4)$-faces.
We uncolor $u$ and $w$, by the argument above, each of the six 3-vertices on pendant 3-faces that adjacent to $v$ and $w$ can be recolored with $1$, then we can recolor $v$ and $w$ with $2$ and $3$, respectively,  and color $u$ with $1$. Next, assume that $f=uvw$ is a subsequent special $(3,5,5)$-face. Then by induction,  the six neighbors of $v$ and $w$ on either previous pendant special $(3,5,5)$-faces  or other pendant special $3$-faces can be recolored with $1$. Thus,  we can recolor $v$ and $w$ with $2$ and $3$, respectively, and then color $u$ with $1$.
\end{proof}

\begin{lem}\label{lem3.11}
Let $v$ be a $4$-vertex with neighbors $v_1, v_2, v_3$ and $v_4$.  Then $v$ cannot be incident to a $(3,4,5^-)$-face $f_1=v_3vv_4$ and $(3, 4, 3, 5^+)$-face $f_2=v_1vv_2w$ with $(b(f_1)\cup b(f_2))\cap C_0=\emptyset$.
\end{lem}

\begin{proof}
Suppose otherwise. By the minimality of $(G, C_0)$, $(G- \{v,v_1,v_2,v_3\}, C_0)$ is superextendable. Properly color $v_1, v_2$ and $v_3$.  Let both $v_1$ and $v_2$ be colored with $1$.  If one of $v_3$ and $v_4$ is colored with $1$, then color $v$ properly; if neither $v_3$ and $v_4$ is colored with $1$, then color $v$ with $1$.  Thus, we may assume that at most one of $v_1$ and $v_2$ is colored with $1$. We can color $v$ with $1$, unless $v_4$ is colored with $1$ and two neighbors (other than $v$ and $v_3$) of $v_4$ are colored with $1$ in which case we recolor $v_4$ properly  and then $v_3$  properly and color $v$ with $1$. So in either case, we have a contradiction.
\end{proof}

\begin{lem}\label{lem3.12}
Let $v\not\in C_0$ be a $5$-vertex with neighbors $v_i$, $0\le i\le 4$.  Then each of the following holds.

{\rm (1)} $v$ cannot be  incident to a $(3,5,3^+)$-face $f=v_4vv_0$ with $d(v_4')\le 4$ and $(b(f)\cup \{v_4'\})\cap C_0=\emptyset$  and adjacent to $3$ pendant special $3$-faces;

{\rm (2)} $v$ cannot be adjacent to $4$ pendant special $3$-faces;

{\rm (3)} $v$ cannot be incident to five $4$-faces from $F_4$ with at least three $(4^-, 3^+, 5, 5^+)$-faces $u_iv_ivv_{i+1}$ such that at most one of $v_i$ and $v_{i+1}$ is incident with a triangle.
\end{lem}

\begin{proof}
(1) Suppose on the contrary that $v$ is incident to a $(3,5,3^+)$-face $f=v_4vv_0$ with $d(v_4')\le 4$ and $(b(f)\cup \{v_4'\})\cap C_0=\emptyset$  and adjacent to $3$ pendant special $3$-faces from $F_3$. By the minimality of $G$, $(G-\{v,v_1,v_2,v_3\}, C_0)$ is superextendable.  By Lemma~\ref{lem3.10}, we recolor $v_1, v_2, v_3$ with $1$. If $v$ cannot be colored, then, without loss of generality,   $v_4$ and $v_0$ are colored with $2$ and $3$, respectively.  Since $d(v_4')\le4$, $v_4'$ can always be nicely colored with $1$ or properly colored with $2$ or $3$. So we recolor $v_4$ by $1$, then color $v$ by $2$, a contradiction.

(2) Suppose on the contrary that $v$ is adjacent to $4$ pendant special $3$-faces, and the pendant neighbors are $v_1,v_2,v_3,v_4$. By the minimality of $G$, $(G- \{v,v_1,v_2,v_3,v_4\}, C_0)$ is superextendable. By Lemma ~\ref{lem3.10}, $v_i$ with $i\in \{1,2,3,4\}$ can be colored with $1$. Then we can properly color $v$, a contradiction.

(3) Suppose on the contrary that such five 4-faces from $F_4$ exist. By the hypothesis, we assume, without loss of generality, that $f_0=u_0v_0vv_1$ and $f_2=u_2v_2vv_3$ are two 4-faces such that $d(u_j)\leq 4$ and at most one of $v_i$ and $v_{i+1}$ is incident with a triangle for $j\in \{0, 2\}$. Let $H=G[\{v_0, v_1\}, \{v_2, v_3\}]$. By Lemma~\ref{lem3.6}, $H\in\mathcal{G}$. By the minimality of $(G, C_0)$,  $(H, C_0)$ is superextendable. We now go back to color the vertices of $G$. We color $v_0$ and $v_1$ with the color of $v_0(v_1)$, and color $v_2$ and $v_3$ with the color of $v_2(v_3)$, and keep the colors of the other vertices. The coloring is valid, unless the following two cases (by symmetry) hold:  (a) both $v_0(v_1)$ and one neighbor of $u_0$ other than $v_0(v_1)$ are colored with $1$ in $H$, or (b) all of $v_0(v_1)$, $v_2(v_3)$ and $v$ are colored with $1$ in $H$ (there may be one neighbor of $u_0$ other than $v_0(v_1)$ is colored with $1$ in $H$ or one neighbor of $u_2$ other than $v_2(v_3)$ are colored with color 1 in $H$).  In the former case,  since $d(u_0)\le 4$, we can recolor $u_0$ properly. In the latter case, if one neighbor of $u_0$ other than $v_0(v_1)$ is colored with $1$ in $H$ or one neighbor of $u_2$ other than $v_2(v_3)$ are colored with $1$ in $H$, we recolor $u_0$ or $u_2$ as in the former case, and then color $v$ properly.
\end{proof}

Lemma~\ref{lem3.12} (3) tells us that if $v\not\in C_0$ is a $5$-vertex,  then it cannot be incident to five $4$-faces from $F_4$ with at least three $(3,3,5,5^+)$-faces. Moreover,  if $v$ is incident to five $4$-faces from $F_4$ with two $(3,3,5,5^+)$-faces, then it cannot be incident to a $(3,4,5,5)$-face from $F_4$.

\begin{lem}\label{lem3.13}
Let $w$ be a $6$-vertex with $h$ pendant special $3$-faces. Then $w$ cannot be incident with a $(3,3,6)$-face $f=uvw$ such that $\min\{d(u'), d(v')\}\le 4$ and $(b(f)\cup\{u',v'\})\cap C_0=\emptyset$ and $h=4$; In addition, if $\max\{d(u'), d(v')\}\le 4$, then $h\le 2$.
\end{lem}

\begin{proof}
 Suppose to the contrary that $w$ is incident to a $(3,3,6)$-face such that $d(u')\le 4$ and $(b(f)\cup \{u',v'\})\cap C_0=\emptyset$ and adjacent to four pendant special $3$-faces. Let  $N(w)=\{u, v, w_1, w_2, w_3, w_4\}$. By the minimality of $G$,  $(G- (N(w)\cup \{w\}), C_0)$ is superextendable. If $h=4$, then by Lemma~\ref{lem3.10}, $w_i$ can be colored by $1$ for $i\in \{1,2,3,4\}$. Since $d(u')\le 4$, $u'$ can be nicely colored. Then we color $u$ by $1$ and color $v$ and $w$ properly to get a desired coloring of $G$, a contradiction.

 Assume that $\max\{d(u'),d(v')\}\le 4$ and $h\ge 3$. As $d(v')\le 4$, $v'$ is nicely colored. Thus we color $v$ with $1$.  Let $w_4$ be the vertex that may not be on a special $3$-face. By the minimality of $G$, $(G-\{u, v, w_1, w_2, w_3\}, C_0)$ is superextendable. As in the proof above, $w_i$ can be colored with $1$ for $i\in \{1,2,3\}$. Since $d(u')\leq 4$, $u$ can be colored with $1$. Then all neighbors of $w$ except $w_4$ are colored with $1$. So we properly color $w$ to get a coloring of $G$, a contradiction again.
\end{proof}

\section{Discharging Procedure}\label{discharging}
In this section, we will finish the proof of the main theorem by a discharging argument.  Let the initial charge of vertex $u\in G$ be $\mu(u)=2d(u)-6$, and the initial charge of face $f\not=C_0$ be $\mu(f)=d(f)-6$ and $\mu(C_0)=d(C_0)+6$. Then
$$ \sum_{u\in V(G)} \mu(u)+ \sum_{f\in F(G)} \mu(f)=0.$$

Let $h$ be the number of pendant special $3$-faces of a vertex $u$.\\

The discharging rules are as follows.

\begin{enumerate}[(R1)]
\item Let $u\not\in C_0$. Then in (R1.1)-(R1.5) $u$ gives charges only to incident or pendant faces that are disjoint from $C_0$, in the following ways: 
\begin{enumerate}
\item[(R1.1)]  $d(u)=4$.

\begin{enumerate}
\item[(R1.1.1)] $u$ gives $\frac{3}{2}$ to each incident $(3,4,5^-)$-face, and $1$ to other incident $3$-faces.

\item[(R1.1.2)] $u$ gives $1$ to the incident $(3,4,3,5^+)$-face if $u$ is incident to a $3$-face, otherwise gives $\frac{1}{2}$ to each incident $4$-face.

\end{enumerate}

\item[(R1.2)] $d(u)=5$

\begin{enumerate}
\item[(R1.2.1)] $u$ gives $1$ to the incident $3$-face if $h=3$, $\frac{3}{2}$ if $h=2$, and $2$ if $h\le 1$.

\item[(R1.2.2)] $u$ gives $1$ to each incident $(3,3,5,5^+)$-face, and if $u$ is incident with a $3$-face then $1$ to each incident $4$-face; if $u$ is not incident with a $3$-face, then $u$ gives $\frac{3}{4}$ to each incident $(3,4,5,5)$-face, and $\frac{2}{3}$ to each other incident $4$-faces.

\end{enumerate}

\item[(R1.3)] $d(u)=6$, then $u$ gives $3$ if $h\le 2$, $\frac{5}{2}$ if $h=3$, and $2$ if $h=4$, to each incident $3$-face.

\item[(R1.4)] $7^+$-vertex gives $1$ to each  pendant $3$-face; $5$- or $6$-vertex gives $1$ to each special pendant $3$-face and $\frac{1}{2}$ to each of the other pendant $3$-faces.

\item[(R1.5)] $6^+$-vertex gives $1$ to each incident $4$-face, and $7^+$-vertex gives $3$ to each incident $3$-face.

\item[(R1.6)] $4^+$-vertex which is incident to a triangle gives $\frac{1}{2}$ to each incident $4$-face from $F_4'$.
\end{enumerate}

\item If $u\in C_0$, then $u$ gives $1$ to each incident $4$-face from $F_4''$ or each pendant face from $F_3$,  $\frac{3}{2}$ to each incident face from $F_3''$ or $F_4'$, and $3$ to each incident face from $F_3'$.

\item $C_0$ gives $2$ to each $2$-vertex on $C_0$, $\frac{3}{2}$ to each $3$-vertex on $C_0$, and $1$ to each $4$-vertex on $C_0$.  In addition, if $C_0$ is a $7$-face with six $2$-vertices, then it gets $1$ from the incident face.

\end{enumerate}

We shall show that each $x\in F(G)\cup V(G)$ other than $C_0$ has final charge $\mu^*(x)\ge 0$ and $\mu^*(C_0)>0$.

First we consider faces.  As $G$ contains no $5$-faces and $6^+$-faces other than $C_0$ are not involved in the discharging procedure, we will first consider $3$- and $4$-faces other than $C_0$. \\

Let $f$ be a $3$-face.  Note that $f$ has initial charge $3-6=-3$. By Lemma \ref{lem3.4}, $|b(f)\cap C_0|\le2$. If $|b(f)\cap C_0|=1$, then  $\mu^*(f)\ge-3+3=0$ by (R2);  if $|b(f)\cap C_0|=2$, then  $\mu^*(f)=-3+\frac{3}{2}\times 2=0$ by (R2).  So we may assume that  $b(f)\cap C_0=\emptyset$.   Let $f=uvw$ with corresponding degrees $(d_1, d_2, d_3)$.

\begin{enumerate}[(1)]
\item  $f$ is a $(3,3,5^-)$-face. By Lemmas~\ref{lem3.8} and \ref{lem3.9}, the neighbors of $3$-vertices on $f$ are either on $C_0$ or have degree at least $5$. In latter case, $f$ is a special pendant $3$-face to them. Thus  each of these neighbors gives $1$ to $f$  by (R2) or (R1.4), plus the $4$- or $5$-vertex on $f$, if exists, gives at least $1$ to $f$ by (R1.1.1) or (R1.2.1),   thus $\mu^*(f)\ge -3+1\times 3=0$.

\item $f$ is a $(3,3,6)$-face. Let $w$ be the 6-vertex of $f$. If $f$ has a pendant neighbor on $C_0$, then  it gets $1$ from $C_0$ by (R2), and  gets at least $2$ from $w$ by (R1.3). Thus, $\mu^*(f)\ge -3+1+2=0$.  We now assume that $f$ has no pendant neighbors on $C_0$.  If each pendant neighbor of the $3$-vertices is of degree at most $4$, then by Lemma~\ref{lem3.13}, $w$ is adjacent to at most $2$ pendant $3$-faces. Thus $w$ gives $3$ to $f$ by (R1.3). Thus, $\mu^*(f)\ge-3+3=0$.  If one pendant neighbor is of degree at most $4$ and the other is of degree at least $5$, then $w$ is adjacent to at most $3$ pendant special $3$-faces by Lemma~\ref{lem3.13}. Thus, $w$ gives $\frac{5}{2}$ to $f$ by (R1.3) and $f$ gets $\frac{1}{2}$ from the pendant neighbor with degree at least $5$ by (R1.4), so $\mu^*(f)\ge -3+\frac{1}{2}+\frac{5}{2}=0$.  If each of the pendant neighbors is of degree at least $5$, then $f$ gets at least $2$ from $w$ by (R1.3) and $\frac{1}{2}$ from each of the pendant neighbors by (R1.4). Thus, $\mu^*(f)\ge -3+2+\frac{1}{2}\times 2=0$.

\item  $f$ is a $(3,3,7^+)$-face. Then $f$ gets $3$ from $w$ by (R1.5). Thus, $\mu^*(f)\ge-3+3=0$.

\item $f$ is a $(3,4,4)$-face. Then $f$ gets $\frac{3}{2}$ from both $v$ and $w$ by (R1.1.1). It follows that $\mu^*(f)\ge-3+\frac{3}{2}\times 2=0$.

\item $f$ is a $(3,4,5)$-face. If $u'\in C_0$, then  $f$ gets $\frac{3}{2}$ from $v$ and $1$ from both $u'$ and $w$ by (R2), (R1.1.1) and (R1.2.1), and hence $\mu^*(f)\ge -3+1+\frac{3}{2}+1=\frac{1}{2}>0$. Thus, we may assume that $u'\not\in C_0$. If $d(u')\in\{3,4\}$, then $w$ is adjacent to at most $2$ pendant special $3$-faces by Lemma~\ref{lem3.12} (1).  Thus,  $f$ gets at least $\frac{3}{2}$ from $w$ by (R1.2.1), and gets $\frac{3}{2}$ from $v$ by (R1.1.1). So, $\mu^*(f)\ge-3+\frac{3}{2}+\frac{3}{2}=0$. If $d(u')\ge 5$, then $f$ gets $\frac{1}{2}$ from $u'$ by (R1.4),   gets $\frac{3}{2}$ from $v$ by (R1.1.1), and gets at least $1$ from $w$ by (R1.2.1). Therefore, $\mu^*(f)\ge-3+\frac{3}{2}+\frac{1}{2}+1=0$.

\item $f$ is a $(3,4^+,6^+)$-face. Then $f$ gets at least $2$ from $w$ by (R1.3) and  gets at least $1$ from $v$ by (R1.1.1) or (R1.2.1). Thus, $\mu^*(f)\ge-3+1+2=0$.

\item $f$ is a $(3,5,5)$-face. If $u'\in C_0$, then  $f$ gets $1$ from $C_0$ by (R2) and gets at least $1$ from both $v$ and $w$ by (R1.2.1), thus $\mu^*(f)\ge-3+1\times3=0$. Thus, assume that $u'\not\in C_0$.  If $d(u')\in\{3,4\}$, then $v$ and $w$ each has at most $2$ pendant special $3$-faces by Lemma~\ref{lem3.12}, thus $f$ gets at least $\frac{3}{2}$ from both $v$ and $w$ by (R1.2.1), therefore $\mu^*(f)\ge-3+\frac{3}{2}\times2=0$. If $d(u')\ge 5$ and $f$ is special, then  $f$ gets at least $1$ from each of the $u', v, w$ by (R1.2.1) and (R1.4), thus $\mu^*(f)\ge-3+1\times3=0$.  If  $d(u')\ge 5$ and $f$ is not special,  then $f$ gets at least $\frac{3}{2}+1$ from $v$ and $w$ and $\frac{1}{2}$ from $u'$ by (R1.4) and (R1.2.1). Thus $\mu^*(f)\ge-3+\frac{3}{2}+\frac{1}{2}+1=0$.

\item $f$ is a $(4^+,4^+,4^+)-$face. Then $f$ gets at least $1$ from each of $u,v,w$ by (R1), thus $\mu^*(f)\ge-3+1\times3=0$.\\
\end{enumerate}

Let $f$ be a $4$-face. Let $f=uvwx$ with corresponding degrees $(d_1,d_2, d_3,d_4)$.  Note that $f$ has initial charge $4-6=-2$. By Lemma \ref{lem3.4}, $|b(f)\cap C_0|\le2$. If $|b(f)\cap C_0|=1$, say $u\in b(f)\cap C_0$, then $u$ gives $\frac{3}{2}$ to $f$ by (R2). By Lemma \ref{lem3.7} each of $u$ and $w$ is incident to a triangle. So  $w$ gives $\frac{1}{2}$ to $f$ by (R1.6). So $\mu^*(f)\ge-2+\frac{3}{2}+\frac{1}{2}=0$. If $|b(f)\cap C_0|=2$, then  $\mu^*(f)=-2+1\times2=0$ by (R2).  Thus,  we now assume that $b(f)\cap C_0=\emptyset$.

Note that by Lemma~\ref{lem3.7}, we only need to consider the following situations.

\begin{enumerate}[(1)]
\item $f$ is a $(3,3,5^+,5^+)$-face. Then $f$ gets at least $1$ from both $w$ and $x$ by (R1.2.2) and (R1.5). Thus, $\mu^*(f)\ge-2+1\times2=0$.

\item $f$ is a $(3,4^+,3,5^+)$-face. Then both $v$ and $x$ are incident to triangles by Lemma~\ref{lem3.7}. Thus $f$ gets at least $1$ from both $v$ and $x$ by (R1.1.2) and (R1.2.2) and (R1.5). It follows that   $\mu^*(f)\ge-2+1\times2=0$.

\item $f$ is a $(3,4^+,4,5^+)$-face.  Then both $v$ and $x$ are incident to triangles by Lemma~\ref{lem3.7}. Thus  $f$ gets at least $1$ from $x$ and $\frac{1}{2}$ from each of $v$ and $w$ by (R1.1.2) and (R1.2.2) and (R1.5). It follows that   $\mu^*(f)\ge-2+1+\frac{1}{2}\times2=0$.

\item $f$ is a $(3,4,5^+,5^+)$-face. If $d(x)=d(w)=5$, then  $f$ gets $\frac{1}{2}$ from $v$ by (R1.1.2), and gets at least $\frac{3}{4}$ from both $w$ and $x$ by (R1.2.2). This implies that $\mu^*(f)\ge-2+\frac{3}{4}\times2+\frac{1}{2}=0$. Otherwise,   $\mu^*(f)\ge -2+\frac{1}{2}+\frac{2}{3}+1>0$ by (R1.1.2), (R1.2.2) and (R1.5).

\item $f$ is a $(3,5^+,5^+,5^+)$-face. Then $f$ gets at least $\frac{2}{3}$ from each of the $5^+$-vertices by (R1.2.2) and (R1.5), thus $\mu^*(f)\ge-2+\frac{2}{3}\times3=0$.

\item $f$ is a $(4^+,4^+,4^+,4^+)$-face. Then $f$ gets at least $\frac{1}{2}$ from each of the four vertices by (R1.1.2), (R1.2.2) and (R1.5). Thus, $\mu^*(f)\ge-2+\frac{1}{2}\times4=0$.\\
\end{enumerate}

Now we consider vertices.  Note that $int(C_0)$ contains no $2^-$-vertices.  For a vertex $u$, let $p$ be the number of $4$-faces incident with $u$, $q$ be the number of pendant $3$-faces adjacent to $u$ and $r$ be the number of $3$-faces incident with $u$.

First let $u\not\in C_0$. Note that if $d(u)=3$ then $u$ is not involved in the discharging process thus $\mu^*(u)=\mu(u)=0$.

\begin{enumerate}[(1)]

\item $d(u)=4$.   If $u$ is not incident with any 3-face, then $u$ is incident with at most four $4$-faces, thus $\mu^*(u)\ge 2-\frac{1}{2}\times4=0$ by (R1.1.2). Thus, we may assume that  $u$ is incident with a 3-face. In this case,  $u$ is incident with a 3-face and at most one $4$-face by our assumption.  If the 3-face is not a $(3,4,5^-)$-face or the $4$-face is not $(3,4,3,5^+)$-face,
then  $u$ gives at most $\max\{1+1, \frac{3}{2}+\frac{1}{2}\}=2$ to the 3-face and the $4$-face by (R1.1.1), (R1.1.2) and (R1.6), thus $\mu^*(u)\ge 2-2=0$. If the 3-face $f_1$ is a $(3,4,5^-)$-face and the $4$-face $f_2$ is a $(3,4,3,5^+)$-face, then one of the vertices on the faces must be on $C_0$ by Lemma~\ref{lem3.11}. By (R1.1.1) and (R1.6), either $u$ only gives at most $\frac{1}{2}$ to the $f_2$ (in this case $|b(f_1)\cap C_0)|\ge1$) or $u$ gives  $\frac{3}{2}$ to $f_1$ and at most $\frac{1}{2}$ to $f_2$ (in this case $|b(f_1)\cap C_0|=0$), thus $\mu^*(u)\ge 2-\frac{3}{2}-\frac{1}{2}=0$.

\item $d(u)=5$.  Assume first that $u$ is incident with a 3-face. In this case,  $u$ is adjacent at most three  pendant 3-faces or at most two incident 4-faces but not both. If $u$ is incident with two 4-faces, then $\mu^*(u)\ge4-2-1\times 2=0$ by (R1.2.1), (R1.2.2) and (R1.6). If $u$ is incident with one 4-face, then $u$ is incident to at most one special 3-face. In this case,  $\mu^*(u)\ge4-2-1-1=0$ by (R1.2.1), (R1.2.2) and (R1.6). If $u$ is not incident with 4-face, let $u$ be adjacent to $h$  pendant special 3-faces. Then $u$ is adjacent to at most $3-h$ pendant 3-faces (not special). If $h=3$, then  $\mu^*(u)\ge 4-1-3\times 1=0$ by (R1.2.1), (R1.2.2) and (R1.6); if $h=2$, then  $\mu^*(u)\ge 4-\frac{3}{2}-2\times 1-\frac{1}{2}=0$ by (R1.2.1), (R1.2.2) and (R1.6);  if $h=1$, then $\mu^*(u)\ge 4-2-2\times \frac{1}{2}-1=0$ by (R1.2.1), (R1.2.2) and (R1.6);  if $h=0$, then $\mu^*(u)\ge 4-2-3\times\frac{1}{2}=\frac{1}{2}>0$ by (R1.2.1), (R1.2.2) and (R1.6).  Thus, we may assume that $u$ is not incident with any 3-face. If $u$ is incident with five 4-faces, then  $u$ is incident with at most two $(3,3,5,5^+)$-faces by Lemma~\ref{lem3.12} (3). Moreover, if $u$ is incident with two $(3,3,5,5^+)$-faces, it cannot be incident with $(3,4,5,5)$-faces. Thus,  if $u$ is incident with at most one $(3,3,5,5^+)$-face, then $\mu^*(u)\ge 4-1-\frac{3}{4}\times4=0$ by (R1.2.2); if $u$ is incident with two $(3,3,5,5^+)$-faces, then  $\mu^*(u)\ge 4-1\times2-\frac{2}{3}\times3=0$ by (R1.2.2). If $u$ is incident with $k(1\le k\le4)$ $4$-faces, then $u$ is adjacent to at most $4-k$ pendant 3-faces. So $\mu^*(u)\ge 4-k-(4-k)=0$ by (R1.2.2) and (R1.4). If $u$ is not incident with any $4$-face, then  it is adjacent to at most three pendent special $3$-faces by Lemma \ref{lem3.12}(2). So $\mu^*(u)\ge 4-3-\frac{1}{2}\times2=0$ by (R1.4).

\item $d(u)=6$.  If $u$ is incident with a 3-face, then $u$ is incident with at most three 4-faces or adjacent to at most four pendant 3-faces but not both. Thus, we need to consider the three cases when  $h\in \{4,3\}$ or $h\leq 2$. So $\mu^*(u)\ge 6-\max\{2+4, \frac{5}{2}+3+\frac{1}{2}, 3+1\times 3\}=0$ by (R1.3),(R1.5) and (R1.6).  If $u$ is not incident with any 3-face, then  $\mu^*(u)\ge 6-1\times 6=0$ by (R1.3) and (R1.5).

\item $d(u)\ge7 $. Since $u$ is incident with at most one 3-face, $r\leq 1$. So $u$ gives at most $1$ to each incident $4$-face and each pendant $3$-face, and $3$ to each of the $r$ incident 3-faces by (R1.4),(R1.5) and (R1.6). Thus $\mu^*(u)\ge 2d(u)-6-(p+q+3r)\ge2d(u)-6-(p+q+2r+1)>2d(u)-6-(d(u)+1)\ge 0$.
\end{enumerate}

Now we consider the case that $u\in C_0$. For $l=3,4$,  each $l$-face $f$ in $G$ satisfies that $|b(f)\cap C_0|\le2$ by Lemma \ref{lem3.4} and furthermore, when $|b(f)\cap C_0|=2$, $f$ and $C_0$ share a common edge.

\begin{enumerate}[(1)]
\item $d(u)=2$.  Then $\mu^*(u)=2\times 2-6+2=0$  by (R3).

\item $d(u)=3$. Then $u$ is not incident with a face from $F_3'$ or $F_4'$.  So $\mu^*(u)\ge -\frac{3}{2}+\frac{3}{2}=0$ by (R2).

\item $d(u)=4$. Assume first that $u$ is incident with a $3$-face $f$. If $f\in F_3'$, then $\mu^*(u)=2-3+1=0$ by (R2) and (R3). If $f\in F_3''$, then it is incident to a $4$-face from $F_4''$ or adjacent to a pendent $3$-face from $F_3$, thus $\mu^*(u)\ge2-\frac{3}{2}-1+1=\frac{1}{2}>0$  by (R2) and (R3). So we may assume that $u$ is not incident with any $3$-face. By  Lemma~\ref{lem3.7}, $u$ is not incident with a face from $F_4'$. Thus we consider that $u$ is incident with $k(\le2)$ $4$-faces from $F_4''$. Then $u$ is adjacent to at most $2-k$ pendent $3$-faces from $F_3$. So $\mu^*(u)\ge2-k-(2-k)=0$ by (R2) and (R3).

\item $d(u)=k\ge 5$.  If $u$ is not incident with any $3$-face, then  $u$ is not incident with a face from $F_4'$ by Lemma~\ref{lem3.7}, so $\mu^*(u)\ge 2k-6-1\cdot (k-2)\ge1>0$ by (R2). Thus, we first assume that  $u$ is incident with a face from $F_3'$.  Let $s$ be the number of $4$-faces in $F_4'$ incident with $u$.  If $s=0$, then $\mu^*(u)\ge2k-6-(k-4)-3\geq 0$ by (R2); if $s\ge1$, then $s\le k-5$, thus $\mu^*(u)\ge2k-6-3-\frac{3}{2}s-(k-s-4)=k-\frac{1}{2}s-5\ge \frac{1}{2}$ by (R2).
  Next, we assume that $u$ is incident with a face from $F_3''$. If $s=0$,  then $\mu^*(u)\ge2k-6-(k-3)-\frac{3}{2}\ge\frac{1}{2}$ by (R2); if  $s\ge1$, then  $s\le k-4$ and $u$ may be incident with at most $k-3-s$ faces from $F_4''$, thus $\mu^*(u)\ge2k-6-\frac{3}{2}s-\frac{3}{2}-(k-s-3)=k-\frac{1}{2}s-\frac{9}{2}\ge \frac{1}{2}s-\frac{1}{2}\ge0$ by (R2).
\end{enumerate}

Now we consider the outer-face $C_0$.  Let $t_i$ be the number of $i$-vertices on $C_0$, then $d(C_0)\ge t_2+t_3+t_4$.  Note that $d(C_0)\in \{3,7\}$.  By (R3),
\begin{align*}\mu^*(C_0)&=d(C_0)+6-2t_2-\frac{3}{2}t_3-t_4\ge d(C_0)+6-\frac{3}{2}(t_2+t_3+t_4)-\frac{t_2}{2}\\
&\ge d(C_0)+6-\frac{3}{2}d(C_0)-\frac{t_2}{2}=6-\frac{d(C_0)}{2}-\frac{t_2}{2}.
\end{align*}

If $d(C_0)=3$ or $t_2\le 5$, then $\mu^*(C_0)\geq 0$. Thus, we may assume that $d(C_0)=7$ and $(t_2,t_3,t_4)\in \{(6, 1, 0), (7,0,0)\}$.  If $t_2=7$, then $G=C_0$ and it is trivially superextendable.  If $t_2=6$ and $t_3=1$, then $C_0$ gets $1$ from the adjacent face which has degree greater than $7$ by (R3), so $\mu^*(C_0)\ge \frac{1}{2}>0$.

We have shown that all vertices and faces have non-negative final charges.  Furthermore, the outer-face has a positive charge, except when $d(C_0)=7$ and $t_2=5$, in which case there must be a face other than $C_0$ having degree more than $7$, but such a face has a positive final charge, as desired.   So $\sum_{x\in V(G)\cup F(G)} \mu^*(x)>0$, a contradiction.

\section*{Acknowledgement}

The authors would like to thank Carl Yerger for him kindness to show us his manuscript and the referees for their valuable comments.

\end{document}